\documentclass[a4paper,11pt,reqno]{amsart}

\usepackage{amssymb}
\usepackage{amsmath}
\usepackage{amsthm}
\usepackage{amsfonts}
\usepackage{bbm}

\usepackage{bookmark}
\usepackage{hyperref}
\hypersetup{pdfstartview={FitH}}

\newcommand{\R}{\mathbb{R}}
\renewcommand{\L}{\mathrm{L}}

\newtheorem{theorem}{Theorem}
\newtheorem{lemma}[theorem]{Lemma}

\DeclareFontFamily{U}{mathx}{\hyphenchar\font45}
\DeclareFontShape{U}{mathx}{m}{n}{
<5> <6> <7> <8> <9> <10>
<10.95> <12> <14.4> <17.28> <20.74> <24.88>
mathx10
}{}
\DeclareSymbolFont{mathx}{U}{mathx}{m}{n}
\DeclareFontSubstitution{U}{mathx}{m}{n}
\DeclareMathAccent{\widecheck}{0}{mathx}{"71}

\numberwithin{equation}{section}

\begin{document}
\title{Singular Brascamp-Lieb inequalities with cubical structure}

\author[P. Durcik]{Polona Durcik}
\address{Polona Durcik, California Institute of Technology, 1200 E California Blvd, Pasadena CA 91125, USA}
\email{durcik@caltech.edu}

\author[C. Thiele]{Christoph Thiele}
\address{Christoph Thiele, Mathematisches Institut, Universit\"at Bonn, Endenicher Allee 60, 53115 Bonn, Germany}
 \email{thiele@math.uni-bonn.de}

\date{\today}

\begin{abstract}
We prove  a singular Brascamp-Lieb inequality, stated in Theorem \ref{mainthm}, with a large group of involutive symmetries. 
\end{abstract}

\maketitle

\section{Introduction}
 
Much research has been devoted in recent years to  Brascamp-Lieb and related inequalities, we refer to \cite{BL}, \cite{BCCT}, 
\cite{BBBF}, \cite{BBCF} and the references therein.
Brascamp-Lieb inequalities are $\L^p$ estimates for certain multilinear forms on functions on 
Euclidean spaces. The forms consist of integrating the tensor product of the input functions over a subspace 
of the direct sum of the domain spaces.  
Following general conventions, we parameterize the subspace of integration by $\R^m$ and write the corresponding Brascamp-Lieb inequality 

 \begin{equation}\label{brasliebineq}
  \Big | \int_{\R^m} \Big ( \prod_{i=1}^n F_i(\Pi_i x)\Big )\, dx \Big | 
 \leq C \prod_{i=1}^n \|F_i\|_{p_i} 
  \end{equation}
with suitable surjective linear maps
$$\Pi_i:\R^m\to \R^{k_i} .$$
Here the constant $C$ is independent of the measurable functions
$F_i$ on $\R^{k_i}$, and integrability on the left-hand side being implied by finiteness of the right-hand side.

It is well understood, under which conditions the Brascamp-Lieb inequality holds.
Bennett, Carbery, Christ, and Tao \cite{BCCT} prove a necessary and sufficient dimensional condition, namely that
  \begin{equation}\label{bcct}
  \dim(V)\le \sum_{i=1}^n \frac 1{p_i} \dim(\Pi_i V)
  \end{equation}
for every subspace $V$ of $\R^m$, with equality if $V=\R^m$. Necessity of inequality \eqref{bcct} is easily seen by testing
the Brascamp-Lieb inequality on certain characteristic functions $F_i$. These functions have minimal support such that the integrand on the left-hand side of \eqref{brasliebineq} is nonzero on a one-neighborhood in $\R^m$ of an arbitrarily large ball in $V$. Necessity of the reverse inequality in case $V=\R^m$ is obtained by using similarly an arbitrarily small ball in $\R^m$.

In this paper, we focus on singular Brascamp-Lieb inequalities. This variant has also seen much development 
in recent years, but still lacks  a general criterion
mirroring  the condition \eqref{bcct}.
A singular Brascamp-Lieb inequality incorporates a Calder\'on-Zygmund kernel
on the left hand side:
\begin{equation}\label{singbraslieb}
\Big | \int_{\R^m} \Big( \prod_{i=1}^n F_i(\Pi_i x)\Big) K(\Pi x)\, dx \Big | \leq C \prod_{i=1}^n \|F_i\|_{p_i} .
\end{equation}
Here $\Pi: \R^m\to \R^k$ is a surjective linear  map,  and by Calder\'on-Zygmund kernel we mean in this paper  a  
 tempered distribution $K$ on $\R^k$ whose
 Fourier transform $\widehat{K}$, called the multiplier associated with $K$, is a measurable function satisfying the symbol estimates
\begin{equation}\label{czsymbol}|\partial^\alpha \widehat{K}(\xi)|\le |\xi|^{-|\alpha|}
\end{equation}
for all $\xi\neq 0$ and all multi-indices $\alpha$ up to suitably large order.  

A necessary condition for the singular Brascamp-Lieb inequality \eqref{singbraslieb} can be obtained by specifying $K$ to be the Dirac delta, that is $\widehat{K}=1$. In this case, \eqref{singbraslieb} can be recognized 
as a classical Brascamp-Lieb inequality \eqref{brasliebineq} with integration over the kernel of $\Pi$. Condition \eqref{bcct} then yields the necessary condition
\begin{equation}\label{newbcct}
  \dim(V)  \le \sum_{i=1}^n \frac 1{p_i} \dim(\Pi_i|_{\ker \Pi} (V))
\end{equation}
for all $V\subseteq \ker \Pi$, with equality if $V=\ker \Pi$.

Lacking a   general necessary and sufficient condition, the theory of singular Brascamp-Lieb inequalities
remains at the stage of a case-by-case study. 
Here, for the first time, we study a sufficiently general family to expose a non-trivial role of the condition \eqref{newbcct}.
We focus on a case that features the following  cubical structure. For a parameter $m\ge 1$ we consider $\R^{2m}$ with coordinates
$$(x_1^0,\ldots ,x_m^0,x_1^1,\ldots ,x_m^1)^T,$$
which we also combine as pair of vectors $(x^0,x^1)^T$ or
we write as vector $x$.
Define the cube  $Q$  to be the set of functions 
$$j:\{1,2,\dots, m\}  \to \{0,1\}.$$ 
For $j\in Q$ define the projection $\Pi_j:\R^{2m} \rightarrow \R^m$ by
\begin{align*}
\Pi_j x = (x_1^{j(1)},x_2^{j(2)},\ldots ,x_m^{j(m)})^T.
\end{align*}
 Our main theorem states that for these particular projections $\Pi_j$
 and for the exponents $p_j=2^m$, inequalities \eqref{newbcct} provide a sufficient condition on an otherwise arbitrary surjective  linear map $\Pi:\R^{2m}\to \R^m$ for the singular Brascamp-Lieb inquality to hold.

\begin{theorem} 
\label{mainthm}
Given $m\ge 1$, there is an $N\ge 0$ such that for all 
surjective linear maps $\Pi:\R^{2m}\to \R^m$
the following are equivalent.
\begin{enumerate}
\item For all subspaces $V\subset {\ker  \Pi}$ we have
\begin{equation}
\label{cubebcct}
  \dim(V)  \le \sum_{j\in Q} 2^{-m} \dim(\Pi_j |_{\ker \Pi} (V)),
\end{equation}
with equality if $V=\ker \Pi$.
\item
For all $j\in Q$, the composed map $\Pi \Pi_j^T$  
is regular.
\item  There is a constant $C$ such that for all Calder\'on-Zygmund 
kernels $K$ satisfying the symbol estimates  \eqref{czsymbol} for all multi-indices up to degree $N$, and for all
  tuples of Schwartz functions $(F_j)_{j\in Q}$ we have
 \begin{equation}\label{maininequality}
\Big| \int_{\R^{2m}} \Big( \prod_{j\in Q} F_j(\Pi_j x)\Big ) K(\Pi x)\, dx\, \Big | \le C \prod_{j\in Q} \|F_j\|_{2^m} .
\end{equation}
\end{enumerate}
 \end{theorem}

 Condition (1) of Theorem \ref{mainthm} is the necessary condition
 derived from that of Bennett, Carbery, Christ, and Tao. In the
 present setting it can immediately be simplified.
 For $V=\ker \Pi$, the left-hand side of \eqref{cubebcct} is at least $m$, while  the right-hand side is at most $m$, because each summand is at most $2^{-m}m$ and there are $2^m$ summands.
 Assuming that inequality \eqref{cubebcct} holds for this $V$, we
 conclude actual equality for this $V$. We further conclude that
 the restriction of $\Pi_j$ to $\ker \Pi$ is injective for each $j$,
 and therefore equality in \eqref{cubebcct} holds for all
 subspaces $V$  of $\ker \Pi$. Thus condition (1) in Theorem \ref{mainthm} is equivalent to the single instance with $V=\ker \Pi $, which in turn
 is equivalent to all $\Pi_j$ being injective on $\ker \Pi $.

 It is now easy to see that conditions (1) and (2) in Theorem \ref{mainthm}
 are equivalent.
 Namely, let $j$ and $l$ be any two opposite corners of the
 cube. Then the range of $\Pi_j^T$ is obviously the kernel of
 $\Pi_l$ and $2^m$ dimensional. Hence regularity of $\Pi \Pi_j^T$
 is the same as injectivitiy of $\Pi$ on the kernel of $\Pi_l$.
 By the above discussion, conditions (1) and (2) are equivalent.

 We have already argued that (3) implies (1), hence the main content of
 the Theorem \ref{mainthm} is that (2) implies (3).

 While the projections $\Pi_j$ of Theorem \ref {mainthm} may appear rather particular, they provide no loss of generality up
 to change of variables after fixing
 their combinatorial datum, that is the set of integer tuples $(\dim(\Pi_j(V))_{j\in Q}$ with $V$ a subspace of $\R^{2m}$. For each $1\le i\le m$, there exist one-dimensional subspaces $V$ and $W$ of $\R^{2m}$, each spanning a certain standard coordinate axis, such that  $\dim(\Pi_j(V))=j(i)$ and $\dim(\Pi_j(W))=1-j(i)$ for all $j$.  
 Conversely, consider any collection of linear maps $(\widetilde{\Pi}_j)_{j\in Q}$ defined on $\R^{2m}$ with $m$ dimensional range, such that for each  $1\le i\le m$ there are spaces $V$ and $W$ with combinatorial datum analoguous as above. Then  these spaces necessarily are one dimensional and together span $\R^{2m}$. A suitable linear transformation of $\R^m$ will turn these vector spaces into the standard coordinate axes. Together with a suitable choice of basis for the range of each of the maps $\widetilde{\Pi}_j$, these maps will be identified as the above maps $\Pi_j$.

 The role of the cubical structure of the form in this theorem 
 is to allow for a symmetrization process in the tuple of functions $F_j$. Indeed, the main Lemma \ref{indstep} stated in Section \ref{sec:inductivest} is an induction over the number 
 of axis parallel symmetry planes of this cube that the tuple $F_j$ respects, in the sense of \eqref{reflectct}.
 This symmetrization procedure, sometimes called twisted technology, originates  in a series of papers such as \cite{vk:tp}, \cite{vk:bell}, \cite{pd:Lp}.
 Theorem \ref{mainthm} in the case $m=2$ generalizes estimates in \cite{pd:L4} and \cite{dkst}.

Further generalizations of Theorem \ref{mainthm} appear desirable, but  are beyond the scope of the present paper, except for a mild vector-valued generalization  in Lemma \ref{indstep}.
Most naturally, one could seek an extension to other exponents $p_j$ and ask for an optimal range of exponents. One may also seek generalizations in which the index set is a subset of the cube. This can sometimes be achieved by setting some functions $F_j$ constantly equal to one, provided one has bounds with $p_j=\infty$.
A further question concerns the exact dependence on $\Pi$ 
of the bounds in the theorem.

To elaborate some of the difficulties in the absence of the cubical structure, we briefly discuss a singular Brascamp-Lieb integral
with three input functions. We take $m=4$ and $k=2$, $k_i=2$ and  $p_i=3$ for $i=1,2,3$. 
The projections $\Pi$ and $\Pi_i$ are then given by $2\times 4$ matrices, which we write as block matrices $(B \  A)$ and $(B_i \ A_i )$ with quadratic blocks.
Choosing coordinates suitably on domain and range of $\Pi$, we may assume that
$$\Pi=(\begin{array}{cc} 0 &  I\end{array})$$
with the identity matrix $I$. In order to not violate \eqref{newbcct} with $V$ equal to $\ker\Pi$, the matrices $B_i$ need
to be regular. Changing coordinates on the range of $\Pi_i$, we may assume $B_i=I$ for each $1\le i\le 3$.
 Warchalski, in his PhD thesis  \cite{mw}, classifies the possiblilities for the remaining parameters
$A_1$, $A_2$, $A_3$ into nine cases. 
Most cases can be normalized such that $A_1=0$ and $A_2=I$,
leaving only $A_3$ as indetermined matrix.
A trivial case occurs if $A_3=0$ or $A_3=I$, this results in a reduction of the complexity of the integral by combining
$F_3$ with one of the other functions by a pointwise product. The case that all eigenvalues of $A_3$ are different 
from $0$ and $1$ is the generic two dimensional version
of the bilinear Hilbert transform \cite{lt}. The known proofs of the singular Brascamp-Lieb inequality in this case require the technique of time-frequency analysis, which is somewhat
different from the technique in the present paper. The case that one eigenvalue of $A_3$ is equal to $0$ or $1$ and the
other eigenvalue is different from $0$ and $1$ is an interesting hybrid case discussed in \cite{dt:2dbht}. The case when $A_3$ has both $0$ and $1$ as eigenvalue
is called the twisted paraproduct and is an instance of the forms in Theorem \ref{mainthm} with $m=2$, albeit with the fourth
function set constant equal to $1$. 
The only case in Warchalski's thesis  where the singular Brascamp-Lieb inequality is not known to hold is the one where the first columns of all three
matrices $A_1,A_2,A_3$ vanish, while the second columns  are $(0,0)^T$, $(0,1)^T$, $(1,0)^T$, respectively.
Thanks to the vanishing first columns, one variable integrates out trivially and one reduces to a one-dimensional Calder\'on-Zygmund kernel.
The paradigmatic example in this case is the conjectured inequality
$$\Big | \int_{\R^3} F(x,y) G(y,z) H(z,x) \frac 1{x+y+z} \, dxdydz\Big | \le
C \|F\|_3\|G\|_3\|H\|_3,$$
 where the left-hand side is called the triangular Hilbert transform. Proving the displayed a priori bound is one of the most
 intriguing open problems in the area of singular Brascamp-Lieb inequalities. Partial progress on this problem can be found in \cite{pz:splx}
 based on the approach in \cite{tt}, and in \cite{dkt}, \cite{ktz:tht}.
 
A more detailed survey of singular Brascamp-Lieb inequalities appears in 
\cite{dt}.

\section{Symmetry considerations and the inductive statement}
\label{sec:inductivest}

Theorem \ref{mainthm} will be proven by induction. The inductive statement is the content of  Lemma \ref{indstep} below. In this section we further discuss  certain symmetries of the singular Brascamp-Lieb integrals \eqref{maininequality}, which will be needed in the proof of the inductive statement. 

For the rest of the paper, we  consider a higher-dimensional generalization of the singular Brascamp-Lieb inequality  \eqref{maininequality},  motivated by the  related paper \cite{patterns}  on certain patterns in positive density subsets of the Euclidean space. We  write vectors as column vectors and identify $x\in (\R^{d})^{2m}$ with a vector of vectors as
$$(x_1^0,\ldots ,x_m^0,x_1^1,\ldots ,x_m^1)^T,$$
where $x_1^0,\ldots ,x_m^0,x_1^1,\ldots ,x_m^1 \in \R^{d}$, which we also combine into
a pair of vectors $(x^0,x^1)^T$. 
For $j\in Q$,  let $\Pi_j: (\R^{d})^{2m} \rightarrow (\R^d)^m$  be given by 
\begin{align*}
\Pi_j x = (x_1^{j(1)},x_2^{j(2)},\ldots ,x_m^{j(m)})^T.
\end{align*}
 We define an action of an $m \times m$ matrix $A$ on a vector $y=(y_1,\ldots, y_m)^T\in (\R^{d})^m$ by the Kronecker product of the matrix $A$ with the $m\times m$ identity matrix $I$ 
$$Ay:= (A\otimes I)y,\ \ (Ay)_i = \sum_{j=1}^m a_{ij}y_j  $$
for $1\leq i \leq m$.
In similar fashion, we identify  $\Pi_j: (\R^{d})^{2m} \rightarrow (\R^d)^m$ with  $m\times 2m$ matrices. 
We also restrict attention to those projections 
$\Pi:(\R^d)^{2m}\to (\R^d)^m$ which are given as analoguous block matrix product as
\begin{equation}
\label{bact}\Pi x =(\begin{array}{cc} B &  A\end{array}) x,
\end{equation} 
where $A$ and $B$ are $m\times m$ matrices and $x\in (\R^{d})^{2m}$.
This setup makes our higher-dimensional generalization a 
very simple extension of the one-dimensional theory. 

It is no restriction to assume that all functions $F_j$
in Theorem \ref{mainthm} are real valued.
Schwartz functions in this section will map $F_j:(\R^{d})^m$ to $\R$
and multipliers  $\widehat{K}$ will map $(\R^d)^m \setminus~ \{0\}$ 
to $\mathbb{C}$.

\begin{lemma}[{Symmetries of \eqref{maininequality}}]
	\label{lemma:symmetries}	
	The following  two statements hold.
\begin{enumerate}
\item  \label{sym:scale}
Let $D$ be an $m\times m$ diagonal matrix of rank $m$. Let $\widetilde{D}$ be a $2m\times 2m$ matrix which decomposes into four blocks of size $m\times m$,  the two blocks on the diagonal being $D$ and the  two off-diagonal blocks being $0$.
Let $1\leq p_j \leq \infty$ with $\sum_{j\in Q}  \frac{1}{p_j} =1$.  Then  
\begin{align}\label{FFtilde}
 \int_{(\R^d)^{2m}}\Big ( \prod_{j\in Q} F_j(\Pi_j x) \Big ) K(\Pi x)\, dx =  \int_{(\R^d)^{2m}} \Big ( \prod_{j\in Q} \widetilde{F}_j({\Pi}_{j}x) \Big )\widetilde{K}(\widetilde{\Pi}x) dx
\end{align}  
holds with   
$$\widetilde{F}_j(y):=\det(D)^\frac{d}{p_j} F_j(D y), \  \widetilde{K}(y): = \det(D)^d K(Dy),\ \widetilde{\Pi} := {D}^{-1} \Pi \widetilde{D}.$$

\item    \label{sym:permute}
Let $P$ be a permutation of $m$ elements, which we also identify with the $m\times m$ matrix in which the $ij$--th entry equals $\delta_{P(i)j}$ in the Kronecker delta notation.
Let $\widetilde{P}$ be a $2m\times 2m$ matrix which decomposes into four blocks of size $m\times m$,  the two blocks on the diagonal being $P$ and the  two off-diagonal blocks being $0$.
 Then  \eqref{FFtilde} holds with  
$$  \widetilde{F}_j(y) :=  F_{j\circ P}(Py),\   
\widetilde{K}(y):=K(Py),\ \widetilde{\Pi}:= P^{-1} \Pi \widetilde{P}.$$ 
\end{enumerate}
\end{lemma}

\begin{proof}
Proof of \eqref{sym:scale}.  Changing variables by $\widetilde{D}$ we have for the left-hand side of \eqref{FFtilde}
\begin{align*}
\det(\widetilde{D})^d \int_{(\R^d)^{2m}}
\Big(  \prod_{j\in Q} F_j(\Pi_j \widetilde{D}x) \Big ) K(\Pi \widetilde{D}x)\, dx. 
\end{align*}
Using $D\Pi_j=\Pi_j \widetilde{D}$
thanks to the special structure of the projections $\Pi_j$, and using 
$$\det(\widetilde{D})^d=\det(D)^{2d}$$ 
and 
$\sum_{j\in Q} \frac{1}{p_j}=1$,  the previous display equals 
\begin{align}\nonumber
& \det(D)^d\int_{(\R^{d})^{2m}} 
\Big( \prod_{j\in Q}  \det(D)^\frac{d}{p_j} F_j(D\Pi_j x)\Big ) K(\Pi \widetilde{D}x)\, dx.
\end{align}
With notation as in \eqref{sym:scale} of the lemma, this becomes the right-hand side of \eqref{FFtilde}.

Proof of (2). We compute  similarly as above
\begin{align*}   
  &\int_{(\R^{d})^{2m}}\Big( \prod_{j\in Q} F_j(\Pi_j x)\Big ) K(\Pi x)\, dx    
    =   \int_{(\R^{d})^{2m}} \Big( \prod_{j\in Q}  {F}_j(\Pi_j\widetilde{P}x) \Big )K(\Pi \widetilde{P}x) \, dx \\
     = & \int_{(\R^{d})^{2m}} \Big( \prod_{j\in Q}  {F}_j(P {\Pi}_{j\circ P^{-1}}x) \Big )\widetilde{K}(\widetilde{\Pi}x) dx
    =  \int_{(\R^{d})^{2m}} \Big( \prod_{j\in Q} {F}_{j\circ P}(P{\Pi}_{j}x)\Big )\widetilde{K}(\widetilde{\Pi}x) dx,
\end{align*}
with notation as in  \eqref{sym:permute} of the lemma.
 \end{proof}

To prove Theorem \ref{mainthm} it suffices to consider the singular Brascamp-Lieb integral
\begin{equation}\label{sblA}
\Lambda(K,A):=\int_{(\R^{d})^{2m}} \Big ( \prod_{j\in Q} F_j(\Pi_j x)\Big ) K((I\ A)x)\, dx.
\end{equation}
This is justified as follows.
Note that if $K$ is a Calder\'on-Zygmund kernel on $\R^{dm}$, 
then so is a certain nonzero scalar multiple of 
$\widetilde{K}$ defined by
$$\widetilde{K}(Eu)=K(u)$$
for some regular matrix $E$. Hence
$$\widetilde{K}(\widetilde{\Pi}u)=K(\Pi u)$$
with $\widetilde{\Pi}:=E\Pi$.
Regularity of $\widetilde{\Pi}\Pi_j ^T$ is equivalent to
regularity of ${\Pi}\Pi_j ^T$, so we may use this flexibility 
to replace the matrix
$(B \  A )$ in \eqref{bact} by $(EB \ EA)$ and therefore assume that 
the matrix $B$ is diagonal and idempotent.
Regularity of all  matrices $\widetilde{\Pi}\Pi_j ^T$ then 
requires $B$ to be the identity matrix.

Let $1\le i\le m$ act by reflection $j\mapsto i*j$ on the cube $Q$,
where
$$(i* j)(k):= j(k)$$
if $i\neq k$ and 
$$(i *j)(i):= 1-j(i).$$
Denote the Gaussian on $\R^s$ by $g(x):=e^{-\pi |x|^2}$ and write $g_t(x):=t^{-s}g(t^{-s}x)$, where $s$ is to be understood from the context, typically $s=1$,
$d$, $md$, $2md$ or $(2m-2)d$. 
By $\partial_jf$ we denote the $j$-th partial derivative of a function $f$.
Recall that the Hilbert-Schmidt norm $\|A\|_{HS}$ of a matrix $A$ 
is monotone in each of its arguments and dominates the  operator norm
$\|A\|$.

\begin{lemma}[The inductive statement]\label{indstep}
Let $m\ge 1$, $d\geq 1$. Let $0\le l\le m$.  Let $0<\epsilon<1$.
There exists a constant $C$ depending on these parameters 
such that the following holds.

Let $A$ be an $m\times m$ matrix such that
\begin{equation} 
\label{aepsilonct}
|\det ((I \ A)\Pi_j^T) |>\epsilon \quad \textup{and} \quad  \|A\|_{HS}\le \epsilon^{-1}
\end{equation}
for all $1\le j\le m$.
Assume that the first $l$ rows of $A$ coincide with the first
$l$ rows of $-I$.  Let $(F_j)_{j\in Q}$ be a tuple of real valued Schwartz functions with
\begin{align}\label{reflectct}
F_j=F_{i*j} \quad \textup{and} \quad  \|F_j\|_{2^m}=1
\end{align}
for all $j\in Q$ and all $1\le i\le l$. Then the following two estimates hold for \eqref{sblA}.
\begin{enumerate}
	\item Let $K$ be a kernel such that 
	$$|\partial^\alpha \widehat{K}(\xi)|\le |\xi|^{-|\alpha|}$$
	for all multi-indices $\alpha \in \mathbb{N}_0^{dm}$ with $|\alpha|\le 3dm$ and 
\begin{align}
\label{vanishing}
\widehat{K}(\xi_1,\dots, \xi_l, 0,\dots ,0)\equiv 0,
\end{align}
 that is, $\widehat {K}$ vanishes for all $0\neq (\xi_1,\ldots , \xi_m)\in (\R^d)^m$ with $\xi_k=0$ for $k>l$.  Then 
 $$|\Lambda(K,A)|\le C.$$
 
   \item Let $l<i\le m$ and $1\leq k_1,k_2 \leq d$. Let $u \in \R^{dm}$ and let $c\in \L^\infty(0,\infty)$ with $\|c\|_\infty=1$. Let $K$ be the kernel defined by
   $$\widehat{K}(\xi) = \int_0^\infty c_t(u)\, \widehat{g_{i,k_1,k_2}} ((I\ A)^T (t\xi)) e^{2\pi i u\cdot t\xi} \frac{dt}{t},$$
   where $g_{i,k_1,k_2} := \partial_{(i-1)d+k_1} \partial_{(i+m-1)d+k_2} g$. 
  Then  
 $$|\Lambda(K,A)|\le C(1+\|u\|)^{2d(m-1)}.$$
\end{enumerate}
\end{lemma}

Note that the case $l=m$ of (1) is trivially true since then
$K=0$.  On the other hand, (2) is void for $l=m$ since then $l<i\leq m$ does not exist.
The case $l=0$ of (1) implies the desired Theorem 
\ref{mainthm}. We will therefore do an induction on $l$, 
proving Lemma \ref{indstep} assuming that we have already established the lemma for all $l<l'\le m$. 
We will reduce (1) at level $l$ to (2) at the same level $l$, and we will reduce (2) at level $l$ to (1) at the level $l+1$. These two reductions will be performed in the following two sections. 

Note that in the case $m=1$ we are dealing with a one-dimensional Calder{\'on}-Zygmund kernel  and the claim follows from the standard Calder{\'o}n-Zygmund theory. We shall therefore assume $m\geq 2$.

\section{Proof of (1) of Lemma \ref{indstep} }
Consider $m,d,l,\epsilon$ as in Lemma \ref{indstep}. We shall 
prove existence of a constant $C$ such that (1) holds,
under the hypothesis that for the same $m,d,l,\epsilon$ 
there is a constant $C$ such that (2) holds.
  
Let $A$, $(F_j)_j$ and $K$ be given as in (1) of Lemma \ref{indstep}.
Our aim is to decompose ${K}$ into a  convergent sum and integral of kernels defined in $(2)$ of Lemma \ref{indstep}.

We will perform a cone decomposition of $\widehat{K}$. The matrix $A$
determines certain subspaces of $(\R^d)^m$, and each cone will be small enough to avoid some of these subspaces, as elaborated in the following lemma.
In this section we use the notational convention 
$$\xi=(\xi',\xi'')\in \R^{dl}\times \R^{d(m-l)}=\R^{dm}.$$
 
\begin{lemma}\label{geoct}
   
	There is a number $\delta>0$ depending on $\epsilon$, $d$, and $m$, such that the following holds. For $\gamma$ a unit vector in $\R^{d(m-l)}$ define the stick
$$S=\left\{(0,\xi'') \in \R^{dm}: \frac 12 \le \|\xi''\|\le 1, 
\left\|\frac {\xi''}{\|\xi''\|} -\gamma \right\|\le \delta\right\}.$$	
Then there is  $l< i\leq m$ and some $1\leq k_1,k_2 \leq d$ such that for all $\eta\in S$ we have
	\begin{equation}\label{geolemmact}
	\min( |\eta_{ik_1}|,|(A^T\eta)_{ik_2}| ) > \delta,
	\end{equation}
where we write $\eta=(\eta_1,\ldots, \eta_m)^T$, $\eta_j=(\eta_{j1},\ldots, \eta_{jd})^T$ for $1\leq j \leq m$, and analogously we write the coordinates of   $A^T\eta$.  
	
\end{lemma}

 \begin{proof}
 
 We first claim that $S$ contains a point $\xi$ 
such that there is $l< i\leq m$ with 
\begin{align}\label{cancellativecond1ct}
\min(\|\xi_i\|, \|(A^T\xi)_i\|) > 4d\max(1,\|A\|)\delta .
\end{align} 
Assume to get a contradiction that the claim is false.
For every $\xi\in S$   we choose $j\in Q$ such that for
$l < i\le m$ the value of $j(i)$ corresponds to which term on the left hand side of \eqref{cancellativecond1ct} is less than or equal to the right-hand side. Hence we obtain
\begin{align}\label{nonconcellativect}
\|(A_j^T\xi)_i\|\leq 4d\max(1,\|A\|)\delta,
\end{align}
where we have denoted $A_j:=(I\ A)\Pi_j^T.$
By pigeonholing with respect to the $2^m$ elements of $Q$, there exists $j\in Q$ and $S'\subseteq S$ of size $|S'|\geq 2^{-m}|S|$ such that 
\eqref{nonconcellativect} holds for this same $j$ and
all $\xi \in S'$ and $l< i \leq m$.

To obtain a contradiction, we compare the volume of $\Psi S'$,
where $\Psi$ is projection onto the $d(m-l)$ dimensional space spanned by the last components,
with that of the linear image 
$$\Psi A_j^T \Psi^T \Psi S'=\Psi A_j^T S'.$$ 
We obtain
\begin{equation}\label{stickballct}
c(d,m) \delta^{d(m-l)-1}  \leq   |\Psi S'|
=     |\det (\Psi A_j^T \Psi^T) |^{-d}  |\Psi A_j^TS'| 
\leq C(d,m,\epsilon)    \delta^{d(m-l)}  
\end{equation}
with positive constants $c(d,m)$ and $C(d,m,\epsilon)$.
On the left hand side we used the growth in $\delta$ 
of the volume of the stick. On the right hand side 
we used that the first $l$ rows of $A$ equal those of $-I$ and thus 
$$\epsilon< |\det(A_j^T)|=|\det(\Psi A_j^T \Psi^T)|, $$
and we estimated the size of the ball with radius  $\delta$ 
in $\R^{d(m-l)}$ that contains $\Psi A_j^T S'$ by virtue of \eqref{nonconcellativect}.
Choosing $0<\delta<0.1$ small enough depending on $d,m,\epsilon$,
inequality \eqref{stickballct} is a contradiction, thereby proving the claim.

By the triangle inequality,  the $\xi$ obtained via the claim also satisfies 
\begin{equation*}
\min(|\xi_{ik_1}|,|(A^T\xi)_{ik_2}|) \ge  4\max(1,\|A\|)\delta
\end{equation*}
for some 
$1\leq i\leq m$ and $1\leq k_1,k_2 \leq d$. 
To prove the desired lower bound \eqref{geolemmact} for every $\eta\in S$, since $1/2\le \|\eta\|,\|\xi\|  \le 1$, it suffices by scaling to show the analoguous bounds with $2\delta$ on the right-hand side under the assumption that $\|\eta\|=\|\xi\|$.
Then $|\eta-\xi|\leq \delta$ and $|A(\eta-\xi)|\leq \|A\|\delta$. Thus  
\begin{align*}
|(A^T\eta)_{ik_2}| \geq |(A^T\xi)_{ik_2}| - |(A^T(\eta-\xi))_{ik_2}| \geq 4\max(1,\|A\|)\delta - \|A\|\delta > 2\delta ,
\end{align*}
and similarly
\begin{align*}
|\eta_{ik_1}| \geq |\xi_{ik_1}| - |\eta_{ik_1}-\xi_{ik_1}| > 2\delta .
\end{align*}
This completes the proof of Lemma \ref{geoct}.
\end{proof}

We proceed to decompose $K$. Let $\delta$ be as in the above Lemma \ref{geoct}.
Consider a maximal set $\Gamma$ of $\delta/6$-separated vectors 
of unit length in $\R^{d(m-l)}$.  By volume considerations on
the unit sphere, there are at most $C(d,m) \delta^{-d(m-l)}$
elements in $\Gamma$. The balls of radius $\delta/2$ centered around 
these points cover the sphere.

For $\gamma\in \Gamma$, let $\rho_\gamma$ be a smooth nonnegative bump function in $\R^{d(m-l)}$ supported on a ball of radius $\delta$ about $\gamma$ and constant one on ball of radius $\delta/2$ about $\gamma$.
Then evidently $\sum_{\gamma \in \Gamma} \rho_\gamma$ is uniformly bounded below on the unit sphere  and we may consider the partition
of unity of $\R^{md}\setminus \{ \R^{dl} \times \{0\}\}$ by the functions
$$f_\gamma(\xi):=\frac{\rho_\gamma(\xi''/\|\xi''\|)}
{\sum_{\gamma' \in \Gamma} \rho_{\gamma'}(\xi''/\|\xi''\|)}.$$
Note the derivative bounds
$$
 |\partial^\alpha f_\gamma(\xi'')|\leq C_\alpha \|\xi''\|^{-|\alpha|}
$$
 for all $\xi''\neq 0$. We  write
\begin{align*}
\widehat{K}(\xi) =\sum_{\gamma} \widehat{K}(\xi)f_\gamma(\xi'')
= \sum_{\gamma} \widehat{K_\gamma}(\xi).
\end{align*} 
Since the number of summands $K_\gamma$ depends only on $d$ and $m$, 
we may restrict attention to an individual summand and prove
$$ |\Lambda(K_\gamma,A)|\le C. $$

 Let $\psi: \R^{dl} \rightarrow \R$ and  $\phi: \R^{d(m-l)}\rightarrow \R$ be  radial Schwartz functions supported in the annuli
$\{1/2\leq |\eta| \leq 1\}$ in $\R^{dl}$ and $\R^{d(m-l)}$, respectively. We normalize them such that   
\begin{align*}
1=\int_0^\infty \psi(t\xi') \frac{dt}{t} =   \int_0^\infty \phi(t\xi'') \frac{dt}{t} 
=\int_0^\infty \int_0^\infty \psi(s t\xi') \phi(t\xi'')\frac{ds}{s} \frac{dt}{t}
\end{align*}
  for every $\xi',\xi'' \neq 0$.  
Then for each $\xi$ with $\xi''\neq 0$ we decompose $\widehat{K}_\gamma(\xi)$
 according to the small and large values of  $s$ as
 \begin{align}\label{k2ct}
\int_0^\infty \int_1^\infty \widehat{K_\gamma}(\xi)  {\psi}(st\xi') \phi(t\xi'')  \frac{ds}{s} \frac{dt}{t} 
\end{align}
  \begin{align}\label{k1ct}
 + \int_0^\infty \int_0^1 \widehat{K_\gamma}(\xi)  {\psi}(st\xi') \phi(t\xi'')  \frac{ds}{s} \frac{dt}{t} .
  \end{align}

We estimate the effect of the multipliers \eqref{k2ct} and \eqref{k1ct}    separately. For \eqref{k2ct} we integrate in $s$ and note that
 $$\rho(\xi'):=\int_1^\infty {\psi}(s\xi')\frac{ds}{s}$$
extends to a smooth bump function with compact support in $\|\xi'\|<2$.
We then fix  $t$ and rescale the 
corresponding portion of the multiplier back as on the left-hand side of the following display \eqref{definektct}. Moreover, we define the multiplier $\widehat{K_t}$ by
\begin{align}\label{definektct}
\widehat{K_\gamma}(t^{-1}\xi)  {\rho}(\xi') \phi(\xi'') 
=: \widehat{K_t}(\xi) \widehat{g_{i,k_1,k_2}}((I \ A)^T \xi),
\end{align}
where   $g_{i,k_1,k_2}$ is defined in (2) of Lemma \ref{indstep}
for suitable $i,k_1,k_2$.
To make sure that $\widehat{K_t}$ is well defined and well behaved, we need that the second factor on the right-hand side is bounded away from $0$ on the compact support of the left-hand side.  By Lemma \ref{geoct}, there exist   $l+1\leq i\leq m$  and $1\leq k_1,k_2 \leq d$ such that for each   $\xi$ in the support of 
the left-hand side of \eqref{definektct} we have
\begin{align*}
|\xi_{ik_1}| >  \delta \quad \textup{and} \quad |(A^T\xi)_{ik_2}| > \delta .
\end{align*} 
Since   $\widehat{g_{i,k_1,k_2}}$ vanishes only at 
$\xi_{ik_1}=0$ and $(A^T\xi)_{ik_2}=0$, it
is bounded uniformly away from $0$ on the support of the left-hand side of \eqref{definektct}.  Therefore, the function $\widehat{K_t}$ is well defined, smooth, and satisfies some uniform bounds
$$|\partial^\alpha {\widehat{K_t}}(\xi)|\leq C  $$
uniformly in $t$ for all $|\alpha| \leq 3dm$.
We  expand it into its Fourier integral
\begin{align*}
\widehat{K_t}(\xi) = \int_{\R^{dm}} K_t(u)e^{2\pi i u \cdot \xi} du. 
\end{align*}
Integrating by parts,  using the derivative estimates  up to order $3dm$ and bounding the size of the support of $\widehat{K_t}$ by an absolute constant times $\delta^{d m-1}$,  we obtain the bound
\begin{align}\label{bd-fourier1}
|K_t(u)|  \leq C  (1+\|u\|)^{-3dm}.
\end{align}
Combining \eqref{definektct} and \eqref{bd-fourier1}, and rescaling back, we see that 
it  suffices to consider the multiplier
 \begin{align*}
\int_{\R^{dm}} (1+\|u\|)^{-3dm}
\Big ( \int_0^\infty 
\big( K_t(u) (1+\|u\|)^{3dm} \big )
\widehat{g_{i,k_1,k_2}}((I \ A)^T t \xi)
e^{2\pi i u \cdot t \xi}
 \frac{dt}{t}\Big ) du .
\end{align*}

Using (2)  of Lemma \ref{indstep} at level $l$ to estimate the singular Brascamp-Lieb integral
associated with the   multiplier in the bracket for a fixed $u$ and integrating in $u$ we obtain 
the desired bound for \eqref{k2ct}.

It remains to consider the part \eqref{k1ct}. Here we fix $0<s<1$ and consider
$$\widehat{K_s}(\xi):=\
\int_0^\infty  \widehat{K_\gamma}(\xi)  {\psi}(st\xi') \phi(t\xi'')  \frac{dt}{t}. $$
We will prove a bound on $\Lambda(K_s,A)$ that is proportional to $s$, so that
we will be able to integrate against $ds/s$ and obtain a good bound for 
the form associated with \eqref{k1ct}.

Let $D$ be the $m\times m$ diagonal  matrix with $d_{ii}=s^{}$ for $i\leq l$ and 
$d_{ii} = 1$ for $i>l$. 
By (1) of Lemma \ref{lemma:symmetries} we have 
$$\Lambda(K_s,A)=\Lambda(\widetilde{K}_s,\widetilde{A}),$$
where
$$\widetilde{K}_s(\xi)=\det(D)^{d} K_s (D\xi),\  \widetilde{A}=D^{-1}AD.$$
Recall that the first $l$ rows of  $A$    coincide with the first $l$ rows of $-I$,
hence we may view $A$ as lower triangular block matrix relative to the splitting
$$\R^{dl}\times \R^{d(m-l)}.$$ The matrix $\widetilde{A}$ arises by multiplying the non-trivial
off diagonal block by $s\le 1$. Hence  
\begin{align*}
\|\widetilde{A}\|_{HS}\leq \|A\|_{HS} \leq 1/\epsilon, \quad  |\det ((I\ \widetilde{A})\Pi_j)|=|\det ((I\ {A})\Pi_j)| > \epsilon.
\end{align*}
We thus plan to apply (2) of Lemma \ref{indstep} with the matrix $\widetilde{A}$.
We note
$$\widehat{\widetilde{K}_s}(\xi)=\
\int_0^\infty  \widehat{K_\gamma}(D^{-1} \xi)  {\psi}(t\xi') \phi(t\xi'')  \frac{dt}{t} .$$
Now we fix in addition $t$ and rescale similarly to \eqref{definektct}. We set
\begin{equation}
\label{defktsct}
\widehat{K_\gamma}(t^{-1}D^{-1}\xi) \psi(\xi')\phi(\xi'')=:\widehat{{K}_{t,s}}(\xi) 
\widehat{g_{i,k_1,k_2}} ((I\ \widetilde{A})^{T} \xi)  
\end{equation}
with some suitable   $l+1\leq i \leq m$ and $1\leq k_1,k_2 \leq d$ from Lemma \ref{geoct}.
Similarly as in the discussion of \eqref{definektct}, on the compact
support of the left-hand side, $\xi'\sim 1$ and $\xi''\sim 1$, the second factor on the right hand side is bounded below, so the function $\widehat{{K}_{t,s}}$ is well defined.
We now claim that 
$$|\partial^\alpha \widehat{K_{t,s}}(\xi)|\le s C $$
uniformly in $t$ for all multi-indices $\alpha$ up to order $3dm-1$. To see this, we
need to show the analoguous estimate for the left hand side of \eqref{defktsct}.
Applying a partial derivative on the left-hand side, we apply the Leibniz rule
and consider the terms separately.

By  $\partial_{ik}f$ we denote the $((i-1)d+k)$-th partial derivative of a function $f$ on $\R^{md}$, $1\leq i \leq m, 1\leq k\leq d$.  
If one derivative $\partial_{ik}$ with $i>l$, $1\leq k \leq d$,
falls on $\widehat{K_\gamma}(t^{-1}D^{-1}\xi)$, we estimate
$$\partial_{ik}(\widehat{K_\gamma}(t^{-1}D^{-1}\xi))\le C t^{-1}(t^{-1}s^{-1}\|\xi'\|+ t^{-1}\|\xi''\|)^{-1}\le Cs$$
since both $\xi'$ and $\xi''$ can be assumed of unit length. Similarly we estimate
if more than one derivative $\partial_{ik}$ with $i>l$
falls on $\widehat{K_\gamma}(t^{-1}D^{-1}\xi)$. If no such derivative falls on $\widehat{K_\gamma}(t^{-1}D^{-1}\xi)$, then only  partial derivatives $\partial_{ik}$
with $i\le l$ fall on $\widehat{K}(t^{-1}D^{-1}\xi)$. Restricting attention to one such derivative we use the vanishing condition \eqref{vanishing} to obtain with the fundamental theorem of calculus 
\begin{align*}\partial_{ik} (\widehat{{K_\gamma}}(t^{-1}D^{-1} \xi)) & = \partial_{ik}
	\int_0^1  \partial_h  ( \widehat{{K_\gamma}}((ts)^{-1}\xi',h t^{-1}\xi'')  ) \, dh\\
	&
	= (ts)^{-1} \int_0^1  t^{-1}\xi''\cdot (\nabla \partial_{ik} \widehat{K_\gamma}) ((ts)^{-1}\xi',h t^{-1}\xi'') dh,
	\end{align*} 
where $\nabla$ denotes the gradient in the last $d(m-l)$ components. 
The desired estimate now follows through derivative estimates for $\widehat{K_\gamma}$
	with one degree higher than $|\alpha|$, note the gain of the factor $s$ comes from the length of $\xi''$ relative to the length of $\xi$ in the relevant support.

As before, we expand  the Fourier integral
\begin{align*}
\widehat{{K}_{t,s}}(\xi) = \int_{\R^{dm}} {K}_{t,s}(u)e^{2\pi i u \cdot \xi} du
\end{align*}
and we observe the bound
\begin{align*}
|{K}_{t,s}(u)|  \leq C s (1+\|u\|)^{-3dm+1}.
\end{align*}
It  suffices to consider the multiplier
\begin{align*}
\int_{\R^{dm}} (1+\|u\|)^{-3dm+1}
\Big ( \int_0^\infty 
\big( {K}_{t,s}(u) (1+\|u\|)^{3dm-1} \big )
\widehat{g_{i,k_1,k_2}}((I \ \widetilde{A})^T t \xi)
e^{2\pi i u \cdot t \xi}
\frac{dt}{t}\Big ) du.
\end{align*}
We again apply (2) of Lemma \ref{indstep}  at level $l$ and integration in $v$ and $s$ to obtain the desired bound.

\section{Proof of (2) of Lemma \ref{indstep} }

Consider $m,d,l,\epsilon$ as in Lemma \ref{indstep}. We shall 
prove existence of a constant $C$ such that (2) holds,
under the hypothesis that for the same $m,d$ but for
$l$ replaced by $l+1$ and for $\epsilon$ replaced by
possibly much smaller $\tilde{\epsilon}$ depending on $d,m,\epsilon$, 
there is a constant $C$ such that (1) holds.

Let $A$ be as in Lemma \ref{indstep}. Recall  that 
the first $l$ rows of $A$ coincide with the first $l$ rows of $-I$.
We shall assume $l<m$ because the case $l=m$ is void. 
With $i,k_1,k_2$ as in (2) of Lemma \ref{indstep}, we need to estimate
the form associated with  the multiplier
   $$\widehat{K}(\xi) = \int_0^\infty c_t(u)\, \widehat{g_{i,k_1,k_2}} ((I\ A)^T (t\xi)) e^{2\pi i u\cdot t\xi} \frac{dt}{t}.$$
Let us first compute the kernel and the form on the spatial side. We have
\begin{align*}
K((I\ A)x) & = \int_{(\R^d)^{m}} \widehat{K}(\xi)e^{2\pi i \xi\cdot ((I\ A)x)} d\xi\\
& =   \int_0^\infty c_t(u) \int_{(\R^d)^{m}} \widehat{g_{i,k_1,k_2}}((I\ A)^T(t\xi))e^{2\pi i u\cdot t\xi}    e^{2\pi i ((I\ A)^T\xi) \cdot x} d\xi \frac{dt}{t} \\
& =  \int_0^\infty c_t(u) \int_{(\R^d)^{m}}  (g_{i,k_1,k_2})_t(x+(-Ap+ut, p)) dp \frac{dt}{t} ,
\end{align*}
where we write $f_t(\cdot )=t^{-{2dm}}f(t^{-1} \cdot )$ for a function $f$ in dimension $2dm$.
The last equality is verified noting that the right-hand side is the integral of the function
$$y\mapsto (g_{i,k_1,k_2})_t(x+y+(ut,0))$$
over the subspace $\{(-A^T\ I)^Tp:p\in (\R^d)^m\}$, while the left-hand side is the integral
of the Fourier transform of this function over the orthogonal subspace 
$$\{(I\ A)^T\xi: \xi\in (\R^d)^m\}.$$
Using the definition of $g_{i,k_1,k_2}$ and Fubini, we obtain for the associated form $\Lambda(K,A)$
\begin{align}\label{lambdakact}
  \nonumber
 &\int_0^\infty c_t(u) \int_{(\R^d)^m}   \int_{(\R^{d})^{2m-2}} 
    \Big( \int_{\R^d} \Big(\prod_{j(i)=0} F_j(\Pi_j x)\Big)
 (\partial_{k_1} g)_{t}(x_i^0+(-Ap+ut)_i) 
 \, dx_i^0\Big )\\ \nonumber
 & \Big( \int_{\R^d} \Big( \prod_{j(i)=1} F_j(\Pi_j x)\Big)
 (\partial_{k_2} g)_{t}(x_i^1+p_i)
 \, dx_i^1\Big )\\   
 & g_{t}\big ((x^0-Ap+ut)_{h\neq i}, (x^1+p)_{h\neq i}\big )\,
 d ((x^0)_{h\neq i}, (x^1)_{h\neq i} )
 \, dp  \frac {dt}t  .
\end{align}

We next prove a particular case of the desired inequality. The 
particular case is defined by the assumptions $1\le i\le l+1$, $k_1=k_2=:k$,
$c_t=1$ for all $t>0$, $u=0$, and in addition to the symmetries
stated in the lemma, also  $F_{(l+1)*j}=F_j$ for all $j\in Q$,
and the $(l+1)$-st row of $A$ also coincides with the $(l+1)$-st row of $-I$. Note all assumptions are more specific than in (2) of 
Lemma \ref{indstep}, except that we on purpose allow $i\le l$ here.

We then recognize that the first bracket in the last display becomes equal to the second bracket by the conditions on  $i$, $u$, and $A$, and the reflection symmetries of the tuple $(F_j)_{j\in Q}$. The two brackets therefore form a square. As Gaussians are positive and  $c_t(u)$ is positive, the entire form becomes non-negative. This holds for all $1\le i\le l+1$ and all $k$. Therefore, instead of proving bounds for each of these terms, it suffices to estimate the sum of all these terms
over $1\le i\le l+1$ and $k$, which has better algebraic properties.

To identify the good properties of this sum, note it is   associated with the multiplier
$$\widehat{K}_\Sigma(\xi):= \sum_{i=1}^{l+1} \sum_{k=1}^{d} \int_0^\infty \widehat{g_{i,k,k} }((I\ A)^T (t\xi))  \frac{dt}{t}.$$
We will add and subtract $\pi$ from this multiplier. We will estimate by hand the form associated with $\pi$, and we will apply the induction hypothesis to  $\widehat{K}_\Sigma(\xi)-\pi$.

The form associated with $\pi$   on  the spatial side is  $\pi$ times 
$$\Lambda(\delta_0 ,A) = \int_{(\R^d)^{2m}}\Big ( \prod_{j\in Q} F_j(\Pi_j x)\Big )\delta_0((I\ A)x)dx ,$$ 
where  $\delta_0$ denotes the Dirac delta distribution. This is a standard Brascamp-Lieb integral.  Applying the arithmetic-geometric mean  inequality at every point $x$ and pulling the arithmetic mean out of the integral, we bound the last display by
\begin{align*}
2^{-m}\sum_{j\in Q} \Big ( \int_{(\R^d)^{2m}}  F_j(\Pi_j x)^{2^m} \delta_0((I\ A)x)dx \Big ).
\end{align*}
This is an average over $j\in Q$, and it suffices to prove bounds for fixed $j$ as follows
\begin{align*}
& \int_{(\R^d)^{2m}} F_j(\Pi_j(x^0,x^1)^T)^{2^m}\delta_0(x^0 + Ax^1)dx^0dx^1\\
&=\int_{(\R^d)^m}F_j(\Pi_j (-Ax^1,x^1)^T )^{2^m}dx^1 \\
&= |\det (\Pi_j (-A^T\ I)^T  )|^{-d}\|{F}_j\|^{2^m}_{2^m} \le \epsilon^{-d}\|{F}_j\|^{2^m}_{2^m}.
\end{align*}
In the last inequality we used the assumption on $A$ and that the absolute value of the determinant in this display is equal to  
$$ |\det (\Pi_\ell  (I \ A)^T )| = |\det (  (I \ A) \Pi_\ell^T )|\ge \epsilon, $$ 
where  $\ell$ is the   corner of the cube opposite to $j$, that is $j(i)+\ell(i)=1$ for all $i$.   
This completes the bound for the multiplier $\pi$.

To estimate the form associated with $\widehat{K}_\Sigma-\pi$, we apply (1) of Lemma \eqref{indstep} for $l+1$ and $\epsilon$ to a suitably normalized kernel. Most assumptions of (1) are straightforward, the main
difficulty is  the vanishing condition \eqref{vanishing}.
 Using
\begin{align*}
\widehat{g'}(-\eta) = \overline{\widehat{g'}(\eta)}, \quad \widehat{g}(0)=1
\end{align*}
for a Gaussian $g$ on $\R$ and the assumption that the first $l+1$ rows of $A$ are
equal to the first $l+1$ rows of $-I$, we obtain  
\begin{align}\label{kxizeroct}
\widehat{K}_\Sigma(\xi_1,\ldots \xi_{l+1},0,\ldots,0) = \sum_{i=1}^{l+1} \sum_{k=1}^{d}  \int_0^\infty |\widehat{\partial_{k} g}(t\xi_{i})|^2 \Big ( \prod_{j=1, j\neq i}^{l+1}   |\widehat{g}(t\xi_{j})|^2 \Big )  \frac{dt}{t}.
\end{align}
Observe the elementary identity
\begin{align}
\label{heat}
-t\partial_t |\widehat{g}(t\eta)|^2 = \frac{1}{\pi} |\widehat{g'}(t\eta)|^2
\end{align}
valid for a one-dimensional Gaussian. 
Since $\widehat{g}(t\xi_i)$ is a product of one-dimensional Gaussians $\widehat{g}(t\xi_{i1})\cdots \widehat{g}(t\xi_{id})$,   together with the Lebniz rule the identity \eqref{heat}  implies   
\begin{align}
\label{heat2}
-t \partial_t|\widehat{g}(t\xi_i)|^2 = \frac{1}{\pi}\sum_{k=1}^d |\widehat{\partial_k g} (t\xi_i)|^2 .
\end{align}
By   \eqref{heat2}, the fundamental theorem of calculus and another application of 
 the Leibniz rule, we equate \eqref{kxizeroct} with
\begin{align*}
 -\pi  \int_0^\infty t\partial_t \Big ( \prod_{j=1}^{l+1}  |\widehat{g}(t\xi_j)|^2\Big ) \frac{dt}{t} = \pi.
\end{align*}
 This completes verification of \eqref{vanishing} for $\widehat{K}_\Sigma-\pi$ and establishes the desired estimate for the associated form.
 
  To round up the discussion, we present a derivation of the elementary identity \eqref{heat} from the heat equation
$$ \partial_t g_t (s)=  \frac {t}{2\pi } \partial_s^2 g_t(s)$$
and the convolution identity
$$g_{\sqrt{2}t}(s_1-s_0)=\int_{\R} g_{t}(s_1-p) g_{t}(s_0-p)\, dp.$$
Indeed, integrating by parts in $p$ we obtain
\begin{align*}
& \partial_t g_{\sqrt{2}t}(s_1-s_0)\\ 
& =\frac t{2\pi }\int_{\R}  \partial_p^2 g_{t}(s_1-p) g_{t}(s_0-p)\, dp+\frac t{2\pi }\int_{\R}    g_{t}(s_1-p) \partial_p^2 g_{t}(s_0-p)\, dp\\
& = -\frac t{\pi }\int_{\R}  \partial_p  g_{t}(s_1-p) \partial_p g_{t}(s_0-p)\, dp .
\end{align*}
This can be turned into \eqref{heat} by taking the Fourier transform. 

We have completed the estimate of the form associated with \eqref{lambdakact}
in the particular case. It remains to reduce the general case to the particular case. We will reduce to the particular case with $A$ replaced
by different matrices, which may satisfy \eqref{aepsilonct} with different $\tilde{\epsilon}$. These different $\tilde{\epsilon}$
however only depend on $m,d,\epsilon$.

We shall first reduce the general case to the case $i\le l+1$. This is done by a permutation of the coordinates if needed. If $i>l+1$, let $P$ be the involution that switches $i$ and $l+1$.
Applying (2) of Lemma  \ref{lemma:symmetries} reduces the to new data which still satisfy
our assumptions of $(2)$ of Lemma \ref{indstep}. Henceforth we assume $i\le l+1$.

Next, we symmetrize the tuple $F_j$ and the pair $k_1,k_2$. We pull $c_t(u)$ into one of the brackets,  apply Cauchy-Schwarz, and then estimate $c_t(u)$ by a constant. This bounds \eqref{lambdakact} by the geometric mean of 
\begin{align}\nonumber
& \int_0^\infty    \int_{(\R^d)^m}   \int_{(\R^d)^{2m-2}} 
 \Big (\int_{\R^d} \Big(\prod_{j(i)=0} F_j(\Pi_j x)\Big)
(\partial_{k_1}  g)_{t}(x_i^0+(-Ap+ut)_i) 
\, dx_i^0\Big)^2\\
& g_{t}\big ((x^0-Ap+ut)_{h\neq i}, (x^1+p)_{h\neq i}\big)\,
d ((x^0)_{h\neq i}, (x^1)_{h\neq i})
\, dp  \frac {dt}t \label{cs1ct} 
\end{align}
and
\begin{align} \nonumber 
&\int_0^\infty    \int_{(\R^d)^m}   \int_{(\R^d)^{2m-2}} 
 \Big(\int_{\R^d} \Big( \prod_{j(i)=1} F_j(\Pi_j x)\Big)
(\partial_{k_2}  g)_{t}(x_i^1+p_i) 
\, dx_i^1\Big)^2\\
& g_{t}\big((x^0-Ap+ut)_{h\neq i}, (x^1+p)_{h\neq i}\big )\, 
d ((x^0)_{h\neq i}, (x^1)_{h\neq i})
\, dp  \frac {dt}t \label{cs2ct}.
\end{align}

It suffices to bound both terms separately and we begin with
\eqref{cs2ct}. To get rid of $u$, we dominate a non-centered Gaussian by a centered Gaussian  as in
\begin{align*}
g (s+v) \leq 10 g \Big (\frac {s}{2+2\|v\|} \Big ).
\end{align*}
Let $v$ be the vector $u$ with the $i$-th $d$-dimensional component replaced by $0$.
Let $D$ the $m\times m$ diagonal matrix with $d_{hh}=2(1+\|v\|)$ for $h\neq i$, and  $d_{ii} =1$. Using the above
domination we estimate \eqref{cs2ct} by
\begin{align*} 
& \int_0^\infty \int_{(\R^d)^m}   \int_{(\R^{d})^{2m-2}} 
   \Big ( \int_{\R^d} \Big(\prod_{j(i)=0} F_j(\Pi_j x)\Big)
(\partial_{k_2} g)_{t}((D^{-1}x^1)_i+p_i) 
\, dx_i^0\Big )^2\\ \nonumber
& g_{t}\big ((D^{-1}x^0-D^{-1}A p)_{h\neq i}, (D^{-1}x^1+D^{-1}p)_{h\neq i}\big)\,
d ((x^0)_{h\neq i}, (x^1)_{h\neq i} )
\, dp  \frac {dt}t.
\end{align*}
Replacing variables $p$ by $Dp$, $x^0$ by $Dx^0$, $x^1$ by $Dx^1$ and 
using $\widetilde{F}_j$ as in (1) of Lemma~\ref{lemma:symmetries}
turns this into
\begin{align*}
 \det(D)^{2d} & \int_0^\infty \int_{(\R^d)^m}   \int_{(\R^{d})^{2m-2}} 
\Big ( \int_{\R^d} \Big(\prod_{j(i)=0} \widetilde {F}_j(\Pi_j  x)\Big)
(\partial_{k_2} g)_{t}(x^1_i+p_i) 
\, dx_i^0\Big )^2 \\ \nonumber
& g_{t}\big ((x^0-D^{-1}ADp)_{h\neq i}, (x^1+p)_{h\neq i}\big )\,
d ((x^0)_{h\neq i}, (x^1)_{h\neq i} )
\, dp  \frac {dt}t .
\end{align*}

To obtain the desired bound, it suffices to apply the particular case of 
\eqref{lambdakact} with the matrix $D^{-1}\widetilde{A} D$ in place of $A$,
where $\widetilde{A}$ is the matrix whose $i$-th row is that of $-I$ and
whose other rows equal those of $A$.
In particular, the first $l+1$ rows of the   matrix ${D}^{-1}\widetilde{A}{D}$ 
coincide with the first $l+1$ rows of $-I$, and we have
\begin{align*}
\|{D}^{-1}\widetilde{A}{D}\|_{HS}\le \|\widetilde{A}\|_{HS} \leq \|A\|_{HS}+1 \le \epsilon^{-1}+1,
\end{align*}
\begin{align*}
    |\det ((I\ {D}^{-1}\widetilde{A}{D})\Pi_j^T) |  \ge \inf_{\tilde{j}}| \det((I\ {D}^{-1}{A}{D})\Pi_{\tilde{j}}^T) | >\epsilon. 
\end{align*}  
Note that we have the upper bound
$$\det(D)^{2d}\le (1+\|u\|)^{2d(m-1)},$$
which is the additional factor in (2) of Lemma \ref{indstep}.
This concludes the estimate of the term \eqref{cs2ct}.

It remains to estimate the term  \eqref{cs1ct}. We reduce it to
the previous case \eqref{cs2ct} by a $t$-dependent affine linear change of variables  
$$\widetilde{p}= -A p + ut . $$
This reduces  \eqref{cs1ct} to 
\begin{align*}\nonumber
|\det(A)|^{-d} & \int_0^\infty    \int_{(\R^d)^m}   \int_{(\R^d)^{2m-2}} 
\Big (\int_{\R^d} \Big(\prod_{j(i)=0} F_j(\Pi_j x)\Big)
(\partial_{k_1}  g)_{t}(x_i^0+\widetilde{p}_i )
\, dx_i^0\Big)^2\\
& g_{t}\big ((x^0+\widetilde{p})_{h\neq i}, (x^1-A^{-1}\widetilde{p}+A^{-1}ut)_{h\neq i}\big)\,
d ((x^0)_{h\neq i}, (x^1)_{h\neq i} )
\, d\tilde{p}  \frac {dt}t . 
\end{align*}
Interchanging the roles of $0$ and $1$ in the range of $j$
reduces this to the previous case with an additional factor
$|\det(A)|^{-d}$, $A$ replaced by $A^{-1}$ and with $u$ replaced by $A^{-1}u$.
As the first $l+1$ rows of $A^{-1}$ coincide with the first $l+1$ rows of  $-I$, it remains to show the conditions \eqref{aepsilonct} for $A^{-1}$
for some $\tilde{\epsilon}$ depending on $\epsilon, m, d$.

The entries of $A^{-1}$ can be estimated by Cramer's rule
by
$$\|A\|_{HS}^{m-1}\det(A)^{-1}\le \epsilon^{-m}$$
and hence
$$\|A^{-1}\|_{HS}\le m\epsilon^{-m}.$$
Estimating the determinants of 
$(I \ A^{-1})\Pi_j^T$ in absolute value from below is tantamount
to estimating determinants of submatrices of $A^{-1}$ obtained  
by deleting any number of pairs of matching rows and columns. 
Considering  block decompositions with squares on the diagonal
  $$\ A = \left ( \begin{array}{cc}
  A_{11} & A_{12}\\
  A_{21}& A_{22}
  \end{array} \right ),\quad  A^{-1} = \left ( \begin{array}{cc}
X_{11} & X_{12}\\
X_{21}& X_{22}
\end{array} \right ), 
$$
we will show a lower bound on $\det (X_{11})$. The general case, when we delete arbitrary rows and columns of $A^{-1}$ can be deduced similarly after permuting rows and columns. 

Note that $A_{22}$ is invertible, since \eqref{aepsilonct} gives  a lower bound on its determinant when choosing suitable $\Pi_j$.
We successively compute
\begin{align*}
& A_{21}X_{11} +A_{22}X_{21}=0,\\
& A_{12}A_{22}^{-1}A_{21}X_{11} +A_{12}X_{21}=0,\\
& A_{12}A_{22}^{-1}A_{21}X_{11} -A_{11}X_{11}=-I.
\end{align*}
A lower bound on $\det(X_{11})$ follows from an upper bound on the determinant of
$$A_{12}A_{22}^{-1}A_{21} -A_{11}.$$
Such bound follows from an upper bound on the norm of this matrix.
Upper bounds on the norms of $A_{12}$, $A_{22}$, $A_{11}$ are obtained
using the bound on the Hilbert Schmidt norm of $A$, while the bound on the norm
of $A_{22}^{-1}$ uses Cramer's rule as above and the lower bound on the determinant of $A_{22}$. Note finally that  
$$(1+\|A^{-1}u\|)^{2d(m-1)}\le (m^d\epsilon^{-md})^{2dm}(1+\|u\|)^{2d(m-1)},$$
which is up to a constant dominated by the factor in (2) of Lemma \ref{indstep}.

\section*{Acknowledgments}
The authors thank Vjekoslav Kova\v{c} and Kristina Ana \v{S}kreb for inspiring discussions
aided by the bilateral DAAD-MZO
grant {\em Multilinear singular integrals and applications}.
The second author acknowledges support
by the Hausdorff Center for Mathematics and the Deutsche Forschungsgemeinschaft through the Collaborative Research Center 1060. 
The authors thank the anonymous referee for a number of thoughtful comments.

\end{document}